\documentclass[12pt]{amsart}
\usepackage{amsmath, amssymb}
\usepackage{mathrsfs}
\newcommand{\no}[1]{#1}
\renewcommand{\no}[1]{}
\no{\usepackage{times}\usepackage[subscriptcorrection, slantedGreek, nofontinfo]{mtpro}
\renewcommand{\Delta}{\upDelta}}
\usepackage{color}
\usepackage{enumerate,comment}
\date{\today}
\newtheorem{theorem}{Theorem}[section]

\newtheorem{lem}{Lemma}[section]

\theoremstyle{remark}

\newcommand{\bel}{\begin{equation} \label}
\newcommand{\ee}{\end{equation}}

\newcommand{\R}{{\mathbb R}}

\def\phi {\varphi}
\def\epsilon {\varepsilon}

\renewcommand{\leq}{\leqslant}
\renewcommand{\geq}{\geqslant}

\def\beq{\begin{equation}}
\def\eeq{\end{equation}}
\newcommand{\bea}{\begin{eqnarray}}
\newcommand{\eea}{\end{eqnarray}}
\newcommand{\beas}{\begin{eqnarray*}}
\newcommand{\eeas}{\end{eqnarray*}}

\providecommand{\abs}[1]{\left\lvert#1\right\rvert}
\providecommand{\norm}[1]{\left\lVert#1\right\rVert}

\numberwithin{equation}{section}

\title[Recovery of semilinear term]{Recovery of nonlinear terms for reaction diffusion equations from boundary measurements}

\author{Yavar Kian}
\address{Aix Marseille Univ, Universit\'e de Toulon, CNRS, CPT, Marseille, France.}
\email{yavar.kian@univ-amu.fr}

\author{Gunther Uhlmann}
\address{G. Uhlmann, Department of Mathematics\\
       University of Washington\\
       Seattle, WA  98195-4350\\
       USA\\
        and Institute for Advanced Study of the Hong Kong University of Science and Technology}
\email{gunther@math.washington.edu}

\begin{document}

\begin{abstract}
We consider the inverse problem of determining a general semilinear term appearing in  nonlinear   parabolic  equations. For this purpose, we derive a new criterion that allows to prove global recovery of some general class of semilinear terms from lateral boundary measurements of solutions of the equation with initial condition fixed at zero. More precisely, we prove, for what seems to be the first time, the unique and stable recovery of general semilinear terms depending on time and space variables independently of the solution of the nonlinear equation from the knowledge of the parabolic Dirichlet-to-Neumann map associated with the solution of the equation with initial condition fixed at zero.  Our approach is based on the second linearization of the inverse problem under consideration.\\

\end{abstract}

\maketitle

%=========================================================

\section{Introduction}

\subsection{Statement}
Let $\alpha\in(0,1)$ and let $\Omega$ be a bounded and connected $\mathcal C^{2+\alpha}$ domain of $\mathbb{R}^n$, $n\geq 2$.   Fixing $T>0$, we consider  the initial boundary value problem
\bel{eq1}
\left\{
\begin{array}{ll}
\partial_tu-\Delta u+F(t,x,u)=0  & \mbox{in}\ (0,T)\times\Omega ,\\
u= g&\mbox{on}\ (0,T)\times\partial\Omega,\\
u(0,x)=0& x\in\Omega,
\end{array}
\right.
\ee
with  $F\in \mathcal C^1([0,T]\times\overline{\Omega}\times\R)$ a well chosen semilinear terms, $g\in\mathcal C([0,T];\mathcal C^{2+\alpha}(\partial\Omega))\cap \mathcal C^{1+\alpha/2}([0,T];\mathcal C(\partial\Omega))$ a suitable Dirichlet boundary condition. Under suitable assumptions (see Subsection 1.4) there exists $T_1\in(0,T]$ such that the problem \eqref{eq1} admits a unique solution $u\in\mathcal C([0,T_1];\mathcal C^{2+\alpha}(\overline{\Omega}))\cap \mathcal C^{1+\alpha/2}([0,T_1];\mathcal C(\overline{\Omega}))$. The goal of this paper is to prove the unique and stable recovery of the nonlinear term  $F$  from partial knowledge of the parabolic Dirichlet-to-Neumann (DN in short) map $g\mapsto \partial_\nu u_{|(0,T_1)\times\partial\Omega}$  associated with \eqref{eq1}, where  $\nu(x)$ denotes the outward unit normal to $\partial\Omega$ computed at $x \in \partial\Omega$. 

\subsection{Motivations}

Recall that reaction diffusion equations of the form \eqref{eq1} are often used for describing several physical phenomenon with applications in chemistry, biology, geology, physics  and ecology. Such nonlinear equations can be used for describing the spreading of biological populations (e.g. \cite{Fi}), the Rayleigh-B\'enard convection (e.g. \cite{NW}) or models appearing in  combustion theory (e.g. \cite{ZF}). In this context the inverse problem addressed in this paper corresponds to the determination of the underlying physical low of the system, associated with the nonlinear expression  $F$ in \eqref{eq1}, by applying a source and measuring the flux at the lateral  boundary $(0,T)\times\partial\Omega$.

Beside these physical motivations,  there is a natural mathematical motivation for the
study of such inverse problems which are highly nonlinear and ill-posed. Indeed, we consider here a nonlinear problem associated with a nonlinear equation.
\subsection{Known results}

Inverse problems stated for nonlinear equations have received an important interest among the mathematical community these last decades. In this context, the recovery of nonlinear terms corresponds to one of the most challenging problem due to its sever ill-posedness and nonlinearity.
For parabolic equations, the recovery of general semilinear terms has first been adressed by \cite{I1}. Here the author proved the recovery of time independent semilinear terms of the form $F(x,u)$ inside the domain (i.e  $F(x,u)$ with $x\in\Omega$, $u\in \R$) from knowledge of the parabolic DN map associated with a problem similar to \eqref{eq1}, together with the final overdetermination, for solution of the equation having all possible initial conditions. The proof of \cite{I1} is based on the linearization of the inverse problem combined with results of recovery of time-dependent coefficients proved by the same author in \cite{I}. In \cite{CK} the authors proved  that the uniqueness result of \cite{I1} remains true while considering observation given by the parabolic DN map alone for solutions of \eqref{eq1} with all possible initial conditions. Moreover, the authors of \cite{CK} established a stability estimate associated with this problem. It seems that only the recovery of semilinear terms depending on the solution of the equation and space variable at the boundary (semilinear terms of the form $F(x,u)$,  $x\in\partial\Omega$, $u\in\R$)  have been proved from excitation and measurements restricted to the lateral boundary $(0,T)\times\partial\Omega$ for solution of the equation with fixed initial condition  (see e.g. \cite{I2,I3,COY}). We mention also the works of \cite{KR1,KR2} where similar problems have been studied from final time or lateral overdetermination. To the best of our knowledge, there has been no result  so far showing the recovery of general semilinear terms, depending on the space variable $x\in\Omega$,  from lateral measurements of the solution of the associated parabolic equation with initial conditions fixed at zero. 

We mention also   the works of \cite{CK1,EPS2,I3,KKU} devoted to the recovery of quasilinear terms and the works of \cite{FO,HUW,HUW1,IN,IS,KU,KU1,KU2,Ki,LLLS1,LLLS2} devoted to similar problems for hyperbolic  and elliptic equations.

\subsection{Preliminary properties}

Let us fix $T>0$ and two non-decreasing maps $M\in \mathcal C(\R_+;\R_+)$ and $T_1\in \mathcal C(\R_+;(0,T])$. Let us also introduce the following functional space
$$\mathcal H_0:=\{g\in\mathcal C([0,T];\mathcal C^{2+\alpha}(\partial\Omega))\cap \mathcal C^{1+\alpha/2}([0,T];\mathcal C(\partial\Omega)):\ g(0,\cdot)=\partial_tg(0,\cdot)=0\}.$$
We define the set $\mathcal A$ of functions   $F\in \mathcal C^1([0,T]\times\overline{\Omega}\times\R)$ taking values in $\R$ such that the following condition is fulfilled:\\
$(H)$ For all $r>0$ and all $g\in\mathcal H_0$ satisfying 
$$\norm{g}_{\mathcal C([0,T];\mathcal C^{2+\alpha}(\partial\Omega))}+\norm{g}_{\mathcal C^{1+\alpha/2}([0,T];\mathcal C(\partial\Omega))}\leq r$$
the problem \eqref{eq1} admits a unique solution $u\in\mathcal C([0,T_1(r)];\mathcal C^{2+\alpha}(\overline{\Omega}))\cap \mathcal C^{1+\alpha/2}([0,T_1(r)];\mathcal C(\overline{\Omega}))$ satisfying the estimate
\bel{esH} \norm{u}_{\mathcal C([0,T_1(r)];\mathcal C^{2+\alpha}(\overline{\Omega}))}+\norm{u}_{\mathcal C^{1+\alpha/2}([0,T_1(r)];\mathcal C(\overline{\Omega}))}\leq M(r).\ee

According to \cite[Theorem 6.1, pp. 452]{LSU}, \cite[Theorem 2.2, pp. 429]{LSU}, \cite[Theorem 4.1, pp. 443]{LSU}, \cite[Lemma 3.1, pp. 535]{LSU} and \cite[Theorem 5.4, pp. 448]{LSU}, the condition (H) will be fulfilled if $F\in \mathcal C^1([0,T]\times\overline{\Omega}\times\R)$ satisfies the following conditions:\\
(P1) There exist a  non-decreasing function $\mu\in\mathcal C([0,+\infty);[0,+\infty))$ such that
$$ |F(t,x,\lambda)|\leq \mu(|\lambda|),\quad t\in[0,T],\ x\in\overline{\Omega},\ \lambda\in\R.$$
(P2) We have
$$F(0,x,0)=0,\quad x\in\partial\Omega.$$
(P3) There exists two constant $b_1,b_2\geq0$ such that 
$$ F(t,x,\lambda)\lambda\geq -b_1\lambda^2-b_2,\quad t\in[0,T],\ x\in\overline{\Omega},\ \lambda\in\R.$$\\

The above conditions (P1)-(P3) allow to prove existence of a global solution of problem \eqref{eq1} with $T_1(r)=T$, $r>0$. Following some classical results about existence of solutions for \eqref{eq1} (see e.g. \cite{A}) one can prove that condition (H) is fulfilled for some classes of nonlinear terms $F$ satisfying (P1)-(P2)  but not (P3). Solutions of such problems correspond to local solutions of \eqref{eq1} that may blow-up at finite time (see e.g. \cite[Proposition 5.4.1.]{CH}) for which the sign restriction (P3) is not required. In a similar way to \cite{Ki}, in the present paper we work in this  framework with such class of local solutions of \eqref{eq1}.

\subsection{Main results }

We fix  $r>0$, $\epsilon\in(0,1)$, $\delta_1\in(0,T_1(r+\epsilon))$ and $\delta_2>0$ such that there exists $\chi\in\mathcal C^\infty_0(\R_+\times\partial\Omega;[0,+\infty))$ satisfying $\chi=\delta_2$ on $[\delta_1,T]\times\partial\Omega$ and $$\norm{\chi}_{\mathcal C^{1+\alpha/2}([0,T]\times\partial\Omega)}+\norm{\chi}_{\mathcal C([0,T];\mathcal C^{2+\alpha}(\partial\Omega))}=1.$$ Then,  for any $\lambda\in[-r,r]$, $h\in\mathcal H_0$ satisfying
\bel{esss}\norm{h}_{\mathcal C([0,T];\mathcal C^{2+\alpha}(\partial\Omega))}+\norm{h}_{\mathcal C^{1+\alpha/2}([0,T];\mathcal C(\partial\Omega))}\leq \epsilon,\ee
we consider the initial boundary value problem
\bel{eq2}
\left\{
\begin{array}{ll}
\partial_tu-\Delta u+F(t,x,u)=0  & \mbox{in}\ (0,T_1(r+\epsilon))\times\Omega ,\\
u(t,x)= \lambda \chi(t,x)+ h(t,x)&(t,x)\in (0,T_1(r+\epsilon))\times\partial\Omega,\\
u(x,0)=0& x\in\Omega.
\end{array}
\right.
\ee
Assuming that $F\in\mathcal A$, the assumption (H) implies that this problem admits a unique solution $u_{\lambda,h}\in\mathcal C([0,T_1(r+\epsilon)];\mathcal C^{2+\alpha}(\overline{\Omega}))\cap \mathcal C^{1+\alpha/2}([0,T_1(r+\epsilon)];\mathcal C(\overline{\Omega}))$ satisfying the estimate
\bel{ess1} \norm{u_{\lambda,h}}_{\mathcal C([0,T_1(r+\epsilon)];\mathcal C^{2+\alpha}(\overline{\Omega}))}+\norm{u_{\lambda,h}}_{\mathcal C^{1+\alpha/2}([0,T_1(r+\epsilon)];\mathcal C(\overline{\Omega}))}\leq M(r+\epsilon).\ee
We consider $B_\epsilon:=\{h\in\mathcal H_0:\ h\ \textrm{satisfies estimate \eqref{esss}}\}$ and we define, for all $\lambda\in[-r,r]$, the map
$$\mathcal N_{\lambda, r,F}: B_\epsilon\ni h\mapsto \partial_\nu u_{\lambda,h}|_{(0,T_1(r+\epsilon))\times\partial\Omega}\in L^2((0,T_1(r+\epsilon))\times\partial\Omega).$$
Our first main result can be stated as follows

\begin{theorem}
\label{t1}
 For $j=1,2$, let $F_j\in\mathcal A\cap \mathcal C^{1}([0,T]\times\overline{\Omega};\mathcal C^4(\R))$. Assume that the following conditions  
\bel{t1a} F_1(t,x,0)=F_2(t,x,0)=0,\quad t\in(0,T),\ x\in\Omega,\ee
\bel{t1b}\partial_u^2 F_1(t,x,u)\leq0,\quad \partial_u^2 F_1(t,x,-u)\geq0, \quad t\in(0,T),\ x\in\Omega,\ u\in[0,+\infty),\ee
are fulfilled. We fix also $q\in W^{1,\infty}((0,T)\times\Omega)$ such that
\bel{t1d} \max_{j=1,2}\partial_uF_j(t,x,0)\leq q(t,x),\quad (t,x)\in (0,T)\times\Omega.\ee
Then, there exists a constant $a_1>0$, depending only on $\Omega$, $T$, $\delta_1$, $q$, $n$ and $\chi$, such that the   condition
 \bel{t1c} \mathcal N_{\lambda,r,F_1}=\mathcal N_{\lambda,r,F_2},\quad \lambda\in(-r,r)\ee
implies that
$$ F_1(t,x,s)=F_2(t,x,s),\quad (t,x)\in[\delta_1,T_1(r+\epsilon)]\times\Omega,\ s\in[-ra_1,ra_1].$$
\end{theorem}

We will also derive a stability result associated with the uniqueness result of Theorem \ref{t1}. For this purpose, let us first recall that, according to  \cite[Proposition 6.1.]{CK}, for $F\in \mathcal C^{1}(\overline{\Omega}\times[0,T];\mathcal C^3(\R))\cap \mathcal A$, $r>0$ and $\lambda\in[-r,r]$, the map $\mathcal N_{\lambda,r,F_1}$ admits a Fr\'echet derivative at $0$ denoted by $\mathcal N_{\lambda,r,F}'(0)$. We consider the stable recovery of $F$ from the data $\mathcal N_{\lambda,r,F}'(0)$, $\lambda\in[-r,r]$. For this purpose, we start by introducing the functional spaces $$\mathcal K_0^T:=\{H_{|(0,T)\times\partial\Omega}:\ H\in L^2(0,T;H^1(\Omega))\cap H^1(0,T;H^{-1}(\Omega)), H_{|\{0\}\times\Omega}=0\},$$
$$\mathcal K_T^T:=\{H_{|(0,T)\times\partial\Omega}:\ H\in L^2(0,T;H^1(\Omega))\cap H^1(0,T;H^{-1}(\Omega)), H_{|\{T\}\times\Omega}=0\},$$
$$S_0:=\{H\in L^2(0,T;H^1(\Omega))\cap H^1(0,T;H^{-1}(\Omega)): H_{|\{0\}\times\Omega}=0,\ \partial_tH-\Delta H=0\},$$
$$S_T:=\{H\in L^2(0,T;H^1(\Omega))\cap H^1(0,T;H^{-1}(\Omega)): H_{|\{T\}\times\Omega}=0,\ \partial_tH+\Delta H=0\}.$$
According to \cite[Proposition 2.1]{CK1}, for any $h\in \mathcal K_j^T$, $j=0,T$, there exists a unique $H\in S_j$ such that $H_{|(0,T)\times\partial\Omega}=h$. Therefore, we can associate with $\mathcal K_j^T$, $j=0,T$, the norm defined by
$$\norm{H_{|(0,T)\times\partial\Omega}}_{\mathcal K_j^T}^2=\norm{H}_{L^2(0,T;H^1(\Omega))}^2+\norm{H}_{H^{1}(0,T;H^{-1}(\Omega))}^2,\quad H\in S_j.$$
Then, combining \cite[Proposition 6.1.]{CK} with \cite[Proposition 2.3]{CK1}, we deduce that $\mathcal N_{\lambda,r,F_1}'(0)-\mathcal N_{\lambda,r,F_2}'(0)$, $\lambda\in[-r,r]$, can be uniquely extended to a bounded operator from $\mathcal K_0^{T_1(r+\epsilon)}$ to $\mathcal K_{T_1(r+\epsilon)}^*$, where $\mathcal K_{T_1(r+\epsilon)}^*$ denotes the dual space of $\mathcal K_{T_1(r+\epsilon)}^{T_1(r+\epsilon)}$.

Using these properties, we can state our stability result as follows.

\begin{theorem}
\label{t2}
 For $j=1,2$, let $F_j\in\mathcal C^{1}([0,T]\times\overline{\Omega};\mathcal C^4(\R))\cap \mathcal A$ and let conditions \eqref{t1a}-\eqref{t1b} be fulfilled.    Assume also that there exists a positive and non decreasing function $\kappa\in\mathcal C([0,+\infty))$ such that
\bel{t2a}\sum_{j=1}^2\sum_{k=0}^2\norm{\partial_u^kF_j(\cdot,u)}_{W^{1,\infty}((0,T)\times\Omega)}\leq \kappa(|u|),\quad u\in\R.\ee
Then, there exists $a_2>0$, depending on $\kappa(0)$, $T_1(r+\epsilon)$, $\Omega$, $\delta_1$, $n$, $\chi$, such that,  for all $r>0$, we have 
\bel{t2b}\begin{aligned}& \sup_{(t,x)\in (\delta_1,T_1(r+\epsilon))\times\Omega}\sup_{s\in(-ra_2,ra_2)} |F_1(t,x,s)-F_2(t,x,s)|\\
&\leq C_r\sup_{\lambda\in(-r,r)}\ln\left(3+\norm{\mathcal N_{\lambda,r,F_1}'(0)-\mathcal N_{\lambda,r,F_2}'(0)}_{\mathcal B\left(\mathcal K_0^{T_1(r+\epsilon)};\mathcal K_{T_1(r+\epsilon)}^*\right)}^{-1}\right)^{-\theta},\end{aligned}\ee
with $C_r>0$ depending on $\Omega$, $\kappa$, $r$, $M(r+\epsilon)$, $T_1(r+\epsilon)$, $T$, $\chi$, $n$ and $\theta>0$ depending on $n$.
\end{theorem}

\subsection{Comments about our results}

To the best of our knowledge, in Theorem \ref{t1} and \ref{t2} we obtain the first result of unique and stable recovery  of general class of semilinear terms admitting variation inside the domain (semilinear terms of the form $F(x,u)$, $(x,u)\in\Omega\times\R$) from excitation and measurements made only on the lateral boundary $(0,T)\times\partial\Omega$ of solutions with initial values fixed at zero. Indeed, it seems that all other similar results require at least measurements of solutions for an infinite number of different initial conditions (see \cite{I1,I5,CK}). Actually, Theorem \ref{t1} and \ref{t2} give a positive answer to the open problem stated in \cite{I5} (see \cite[Problem 9.6, pp. 296]{I5}) for any time independent semilinear term $F$ satisfying conditions \eqref{t1a}-\eqref{t1b} as well as condition (H) with a function $T_1$ satisfying
\bel{condd}\inf_{r>0}T_1(r)>0.\ee

Let us observe that the results of Theorem \ref{t1} and \ref{t2} extend several comparable results in terms of generality of the semilinear term under consideration from several aspects. First, Theorem \ref{t1} and \ref{t2} seem to be the first results of unique and stable recovery of semilinear terms admitting variation in $t\in(0,T)$ and $x\in\Omega$ independently of the solution. Indeed, it seems that all other results have been stated with semilinear terms of the form $F(x,u)$, $x\in\Omega$, $u\in\R$, (see \cite{CK,I1}) or $F(u)$,  $u\in\R$ (see \cite{I2,I3,COY}).  For instance, Theorem \ref{t1} and \ref{t2} can be applied to the recovery of nonlinear terms associated with physical phenomenon that admits variation in time and space position. Secondly, in contrast to all other similar works that we know, including works for elliptic equations, the results of Theorem \ref{t1} and \ref{t2} can be applied to semilinear terms that are neither Lipschitz (or  with derivative in $u$ uniformly bounded in some suitable sense) with respect to $u\in\R$  (as considered in \cite{IN,IS}) nor lying on some specific classes of semilinear terms admitting an holomorphic extension with respect to $u\in\mathbb C$ (as considered by \cite{FO,KU,KU1,LLLS1,LLLS2}). For instance, fixing $\epsilon_1>0$, $q_\pm,\gamma_\pm\in\mathcal C^1([0,T]\times\overline{\Omega})$  such that $\pm q_\pm\leq0$    and $\gamma_\pm\geq1$, our result can be applied to any semilinear terms $F\in \mathcal C^{1}([0,T]\times\overline{\Omega};\mathcal C^4(\R))$ of the form
\bel{fnon}F(t,x,\pm u)= q_\pm(t,x) (1+|u|)^{\gamma\pm(t,x)},\quad (t,x,u)\in(0,T)\times\Omega\times[\epsilon_1,+\infty)\ee
for which condition (H) and \eqref{t1a}-\eqref{t1b} are fulfilled.

Let us remark that the proof of Theorem \ref{t1} and \ref{t2} are based on a new criteron, stated in \eqref{t1b}, that we impose to the semilinear terms  under consideration. This criterion is fulfilled by nonlinear  terms which are Lipschitz  with respect to $u\in\R$ but, in contrast to the results of \cite{IN,IS}, it is not limited to such class of nonlinear terms. In order to prove Theorem \ref{t1} and \ref{t2} from condition \eqref{t1b}  we proceed to the second linearization of our inverse problem in order to get more information about values of some classes of solutions of \eqref{eq1}. This idea is inspired by the approach of \cite{SuU}, where the second linearization has already been considered for the recovery of some quasilinear terms, as well as some recent developement of inverse problems for nonlinear equations based on multiple linearization (see e.g. \cite{FO,KU,KKU,KU1,KLU,LLLS1,LLLS2}).

In a similar way to several  results like \cite{CK1,Ki,KKU}, the results of Theorem \ref{t1} and \ref{t2} are not stated with the full knowledge of the DN map associated with \eqref{eq1} but only its knowledge on a neighborhood of Dirichlet boundary conditions of the form $\lambda\chi$, with $\lambda\in\R$ a constant.

We mention that the the constants $a_1$ and $a_2$ appearing in the statement of Theorem \ref{t1} and \ref{t2} are explicitly given by formulas \eqref{a1}, \eqref{l4c}. Namely, $a_1$ corresponds to the infimum of the solution $w$ of the problem \eqref{ww} on $(\delta_1,T)\times\Omega$ while $a_2$ denotes  the infimum of the solution $y$ of the problem \eqref{eqy} on $(\delta_1,T)\times\Omega$. We explain also in the proof of Theorem \ref{t1} and \ref{t2} why $a_1,a_2$ are both positive constants. Since here both $\delta_1$ and $\chi$ depends on $T_1(r+\epsilon)$, the constants $a_1,a_2$ depend also on $T_1(r+\epsilon)$. This means that if $T_1(r+\epsilon)$ is lower bounded by a positive constant independent of $r$ the constants $a_1,a_2$ can be chosen independently of the parameter $r>0$. Such phenomenon can occur for instance if problem \eqref{eq1} admits a global solution. In that context, applying Theorem \ref{t1} one can deduce that the data $\mathcal N_{\lambda,r,F}$, $r>0$, $\lambda\in(-r,r)$, determines globally time independent semilinear terms of the form $F(x,u)$, $x\in\Omega$, $u\in\R$, satisfying the conditions \eqref{t1a}-\eqref{t1b}. In the same way, the result of Theorem \ref{t1} implies that the data $\mathcal N_{\lambda,r,F}$, $r>0$, $\lambda\in(-r,r)$, determines uniquely semilinear terms of the form $F(x,u)$, $x\in\Omega$, $u\in\R$, for which there exists a sufficiently small parameter $r_0>0$ such that for all $x\in\Omega$, the map $u\mapsto F(x,u)$ is analytic in $(-\infty,r_0)\cup  (r_0,+\infty)$.

\subsection{Outline} 
This paper is organized as follows. In Section 2, we consider the proof of the uniqueness result  stated in Theorem \ref{t1} while Section 3 is devoted to the proof of the stability result  stated in Theorem \ref{t2}.

\section{The uniqueness result}
This section is devoted to the proof of the uniqueness result stated in Theorem \ref{t1}. For this purpose, we start with the linearization of the inverse problem.
\subsection{Linearization of the inverse problem}

In this subsection we introduce a linearization procedure for the problem \eqref{eq2}. Namely, we recall the first linearization for this inverse problem stated in \cite{CK1,CK,I1}.  For $j=1,2$, let $F_j$ be given by Theorem \ref{t1} and satisfying \eqref{t1a}-\eqref{t1b}. Fixing  $\lambda\in[-r,r]$, $h\in B_\epsilon$ and $s\in(-1,1)$ we consider $u_j=u_{j,\lambda,s,g}$ the solution of
\bel{ppeq2}
\left\{
\begin{array}{ll}
\partial_tu_j-\Delta u_{j}+F_j(t,x,u_j)=0  & \mbox{in}\ (0,T_1(r+\epsilon))\times\Omega ,
\\
u_j(t,x)= \lambda\chi(t,x)+sh(t,x)&(t,x)\in(0,T_1(r+\epsilon))\times\partial\Omega,\\
u_j(0,x)=0 &x\in\Omega.
\end{array}
\right.
\ee
Let us also consider the linear problem
\bel{ppeq3}
\left\{
\begin{array}{ll}
\partial_tu_j^{(1)}-\Delta u_j^{(1)}+V_{j,\lambda}u_j^{(1)}=0  & \mbox{in}\ (0,T_1(r+\epsilon))\times\Omega ,
\\
u_j^{(1)}=h &\mbox{on}\ (0,T_1(r+\epsilon))\times\partial\Omega.\\
u_j^{(1)}(0,x)=0  &x\in\Omega.
\end{array}
\right.
\ee
with 
$$V_{j,\lambda}(t,x)=\partial_uF_j\left(t,x,u_{j,\lambda,s,h}(t,x))|_{s=0}\right).$$
Since $V_{j,\lambda}\in\mathcal C^1([0,T_1(r+\epsilon)]\times\overline{\Omega})$, it is well known (see \cite[Theorem 5.4, pp. 322]{LSU}) that \eqref{ppeq3} admits a unique solution $u_j^{(1)}\in\mathcal C([0,T_1(r+\epsilon)];\mathcal C^{2+\alpha}(\overline{\Omega}))\cap \mathcal C^{1+\alpha/2}([0,T_1(r+\epsilon)];\mathcal C(\overline{\Omega}))$. Then we introduce the parabolic DN map $\Lambda_{V_{j,\lambda}}$ associated with \eqref{ppeq3} given by $\Lambda_{V_{j,\lambda}}:h\mapsto\partial_\nu u_j^{(1)}|_{(0,T_1(r+\epsilon))\times\partial\Omega}$. Following \cite[Proposition 6.1]{CK}, we can prove the following.
\begin{lem}\label{ppl1} For $j=1,2$, $\lambda\in[-r,r]$, $h\in B_\epsilon$, the map $s\longmapsto u_{j,\lambda,s,h}$ admits a Fr\'echet derivative at $s=0$ in the sense of maps taking values in $\mathcal C([0,T_1(r+\epsilon)];\mathcal C^{2+\alpha}(\overline{\Omega}))\cap \mathcal C^{1+\alpha/2}([0,T_1(r+\epsilon)];\mathcal C(\overline{\Omega}))$ and we have 
$$\partial_su_{j,s,\lambda,g}|_{s=0}=u_j^{(1)}.$$
\end{lem}

According to Lemma \ref{ppl1}, \eqref{t1c} implies 
$$\Lambda_{V_{1,\lambda}}=\mathcal N_{\lambda,r,F_1}'(0)=\mathcal N_{\lambda,r,F_2}'(0)=\Lambda_{V_{2,\lambda}},\quad \lambda\in[-r,r].$$
In a similar way to \cite{CK1,CK}, combining this identity with \cite[Theorem 1.1.]{CK}  we deduce that
\bel{t5f}\partial_uF_1\left(\cdot,u_{j,\lambda,s,h}|_{s=0}\right)=V_{1,\lambda}=V_{2,\lambda}=\partial_uF_2\left(\cdot,u_{j,\lambda,s,h}|_{s=0}\right),\quad \lambda\in[-r,r].\ee
For all $\lambda\in[-r,r]$, let us consider $v_{j,\lambda}\in\mathcal C([0,T_1(r+\epsilon)];\mathcal C^{2+\alpha}(\overline{\Omega}))\cap \mathcal C^{1+\alpha/2}([0,T_1(r+\epsilon)];\mathcal C(\overline{\Omega}))$ the solution of 
\bel{ppeq4}
\left\{
\begin{array}{ll}
\partial_tv_{j,\lambda}-\Delta v_{j,\lambda}+F_j(t,x,v_{j,\lambda})=0  & \mbox{in}\ (0,T_1(r+\epsilon))\times\Omega ,
\\
v_{j,\lambda}(t,x)= \lambda\chi(t,x)&(t,x)\in(0,T_1(r+\epsilon))\times\partial\Omega,\\
v_{j,\tau}(0,x)=0  &x\in\Omega.
\end{array}
\right.
\ee
Then, the identity \eqref{t5f} can be rewritten as
\bel{t5g}\partial_uF_1\left(t,x,v_{1,\lambda}(t,x)\right)=\partial_uF_2\left(t,x,v_{2,\lambda}(t,x)\right),\quad (t,x,\lambda)\in(0,T_1(r+\epsilon))\times\Omega\times[-r,r].\ee
Now let us consider $v_{j,\lambda}^{(1)}\in\mathcal C([0,T_1(r+\epsilon)];\mathcal C^{2+\alpha}(\overline{\Omega}))\cap \mathcal C^{1+\alpha/2}([0,T_1(r+\epsilon)];\mathcal C(\overline{\Omega}))$  solving the  problem 
\bel{ppeq5}
\left\{
\begin{array}{ll}
\partial_tv_{j,\lambda}^{(1)}-\Delta v_{j,\lambda}^{(1)}+V_{j,\lambda}v_{j,\lambda}^{(1)}=0  & \mbox{in}\ (0,T_1(r+\epsilon))\times\Omega ,
\\
v_{j,\lambda}^{(1)}(t,x)= \chi(t,x)&(t,x)\in(0,T_1(r+\epsilon))\times\partial\Omega,\\
v_{j,\lambda}^{(1)}(0,x)=0  &x\in\Omega.
\end{array}
\right.
\ee
Repeating the arguments used in Lemma \ref{ppl1}, one can check that $\partial_\lambda v_{j,\lambda}=v_{j,\lambda}^{(1)}$, $j=1,2$, $\lambda\in(-r,r)$. Moreover, the identity \eqref{t5f} implies that $v_{1,\lambda}^{(1)}=v_{2,\lambda}^{(1)}$, since they both solves the same initial boundary value problem which admits a unique solution. Combining this with the fact that condition \eqref{t1a} and the uniqueness of solution of \eqref{ppeq4} imply that $v_{1,0}=v_{2,0}=0$, we deduce that
$$v_{1,\lambda}=\int_0^\lambda v_{1,\tau}^{(1)}d\tau=\int_0^\lambda v_{2,\tau}^{(1)}d\tau=v_{2,\lambda},\quad \lambda\in[-r,r].$$
According to this identity, fixing $G:=\{(t,x,v_{1,\lambda}(t,x)):\ x\in\Omega,\ t\in(0,T_1(r+\epsilon)),\ \lambda\in[-r,r]\}$, the consequence of the first linearization of the inverse problem stated in Theorem \ref{t1} can be stated as follows 
$$ F_1(t,x,s)=F_2(t,x,s),\quad (t,x,s)\in G.$$
In view of this identity, the proof of Theorem \ref{t1} will be completed if we prove that exists a constant $a_1>0$, depending only on $\Omega$, $\delta_1$, $T$, $q$ and $\chi$, such that $(\delta_1,T_1(r+\epsilon))\times\Omega\times[-a_1r,a_1r]\subset G$.  For this purpose, we introduce $v_{1,\lambda}^{(2)}$ solving the following problem
\bel{ppeq6}
\left\{
\begin{array}{ll}
\partial_tv_{1,\lambda}^{(2)}-\Delta v_{1,\lambda}^{(2)}+V_{j,\lambda}v_{1,\lambda}^{(2)}=-\partial_u^2F_1(t,x,v_{1,\lambda})\left(v_{1,\lambda}^{(1)}\right)^2  & \mbox{in}\ (0,T_1(r+\epsilon))\times\Omega ,
\\
v_{1,\lambda}^{(2)}=0 &\mbox{on}\ (0,T_1(r+\epsilon))\times\partial\Omega,\\
v_{1,\lambda}^{(2)}(0,x)=0  &x\in\Omega.
\end{array}
\right.
\ee
 The second linearization of problem \eqref{ppeq4} will give us the following result.

\begin{lem}\label{ppl2} The map $\lambda\longmapsto v_{1,\lambda}^{(1)}$ admits a Fr\'echet derivative in $(-r,r)$ and we have 
\bel{ppl2a}\partial_\lambda v_{1,\lambda}^{(1)}=v_{1,\lambda}^{(2)},\quad \lambda\in(-r,r),\ee
in the sense of functions taking values in $\mathcal C([0,T_1(r+\epsilon)];\mathcal C^{2+\alpha}(\overline{\Omega}))\cap \mathcal C^{1+\alpha/2}([0,T_1(r+\epsilon)];\mathcal C(\overline{\Omega}))$.
\end{lem}
\begin{proof} We start by proving that  the map $(-r,r)\ni\lambda\longmapsto v_{1,\lambda}^{(1)}\in\mathcal C([0,T_1(r+\epsilon)];\mathcal C^{2+\alpha}(\overline{\Omega}))\cap \mathcal C^{1+\alpha/2}([0,T_1(r+\epsilon)];\mathcal C(\overline{\Omega}))$ is continuous. For this purpose, we fix $\lambda\in(-r,r)$, $\delta\in(|\lambda|-r,r-|\lambda|)\setminus\{0\}$ and consider
$y=v_{1,\lambda+\delta}^{(1)}-v_{1,\lambda}^{(1)}$. It is clear that $y$ solves

$$\left\{
\begin{array}{ll}
\partial_ty-\Delta y+V_{j,\lambda}y=G_{\delta,\lambda}  & \mbox{in}\ (0,T_1(r+\epsilon))\times\Omega ,
\\
y=0 &\mbox{on}\ (0,T_1(r+\epsilon))\times\partial\Omega,\\
y(0,x)=0  &x\in\Omega.
\end{array}\right.$$
with
$$G_{\delta,\lambda}=(V_{1,\lambda}-V_{1,\lambda+\delta})v_{1,\lambda+\delta}^{(1)}.$$
From now on we denote by $\norm{\cdot}_{2+\alpha,1+\alpha/2}$ the norm of $\mathcal C([0,T_1(r+\epsilon)];\mathcal C^{2+\alpha}(\overline{\Omega}))\cap \mathcal C^{1+\alpha/2}([0,T_1(r+\epsilon)];\mathcal C(\overline{\Omega}))$ defined by
$$\norm{f}_{2+\alpha,1+\alpha/2}=\max\left(\norm{f}_{\mathcal C([0,T_1(r+\epsilon)];\mathcal C^{2+\alpha}(\overline{\Omega}))},\norm{f}_{\mathcal C^{1+\alpha/2}([0,T_1(r+\epsilon)];\mathcal C(\overline{\Omega}))}\right).$$
In view of \eqref{ess1}, one can check that
\bel{pl2b} M_1(\lambda)=\sup_{\delta\in(|\lambda|-r,r-|\lambda|)}\left( \norm{v_{1,\lambda+\delta}^{(1)}}_{2+\alpha,1+\alpha/2}+\norm{v_{1,\lambda+\delta}}_{2+\alpha,1+\alpha/2}\right)<\infty.\ee
Moreover, the mean value theorem implies that 
$$\begin{aligned}&\norm{V_{1,\lambda}-V_{1,\lambda+\delta}}_{L^\infty((0,T_1(r+\epsilon))\times\Omega)}\\
&\leq \left(\sup_{(t,x)\in(0,T_1(r+\epsilon))\times\Omega}\sup_{s\in[-M_1(\lambda),M_1(\lambda)]}|\partial_u^2 F_1(t,x,s)|\right) \norm{v_{1,\lambda+\delta}-v_{1,\lambda}}_{L^\infty((0,T_1(r+\epsilon))\times\Omega)}\\
&\leq C(\lambda)\norm{v_{1,\lambda+\delta}-v_{1,\lambda}}_{L^\infty((0,T_1(r+\epsilon))\times\Omega)}\end{aligned}$$
In the same way, we can prove that, for all $k\in\mathbb N$ and all $\beta\in\mathbb N^n$ such that $|\beta|+k=1$,   we have
$$\begin{aligned}&\norm{\partial_t^k\partial_x^\beta V_{1,\lambda}-\partial_t^k\partial_x^\beta V_{1,\lambda+\delta}}_{L^\infty((0,T_1(r+\epsilon))\times\Omega)}\\
&\leq 2\left(\sup_{(t,x)\in(0,T_1(r+\epsilon))\times\Omega}\sup_{s\in[-M_1(\lambda),M_1(\lambda)]}|(\partial_u^2 \partial_t^k\partial_x^\beta F_1(t,x,s)|+|\partial_u^2 F_1(t,x,s)|)\right) \norm{v_{1,\lambda+\delta}-v_{1,\lambda}}_{W^{1,\infty}((0,T_1(r+\epsilon))\times\Omega)}\\
&\ \ +M_1\left(\sup_{(t,x)\in(0,T_1(r+\epsilon))\times\Omega}\sup_{s\in[-M_1(\lambda),M_1(\lambda)]}|\partial_u^3 F_1(t,x,s)|\right) \norm{v_{1,\lambda+\delta}-v_{1,\lambda}}_{L^\infty((0,T_1(r+\epsilon))\times\Omega)}\\
&\leq C(\lambda)\norm{v_{1,\lambda+\delta}-v_{1,\lambda}}_{L^\infty((0,T_1(r+\epsilon))\times\Omega)}.\end{aligned}$$
Combining this with  Lemma \ref{ppl1}, we deduce that
\bel{pl2c}\lim_{\delta\to0}\norm{G_{\delta,\lambda}}_{\mathcal C^1([0,T_1(r+\epsilon)]\times\overline{\Omega})}=0.\ee
In addition, applying  \cite[Theorem 5.2, p.p. 320]{LSU}, we get
$$\norm{v_{1,\lambda+\delta}^{(1)}-v_{1,\lambda}^{(1)}}_{2+\alpha,1+\alpha/2}=\norm{y}_{2+\alpha,1+\alpha/2}\leq C(\norm{G_{\delta,\lambda}}_{\mathcal C([0,T_1(r+\epsilon)];\mathcal C^{\alpha}(\overline{\Omega}))}+\norm{G_{\delta,\lambda}}_{\mathcal C^{\alpha/2}([0,T_1(r+\epsilon)];\mathcal C(\overline{\Omega}))}).$$
Combining this with \eqref{pl2c}, we deduce that $(-r,r)\ni\lambda\longmapsto v_{1,\lambda}^{(1)}\in\mathcal C([0,T_1(r+\epsilon)];\mathcal C^{2+\alpha}(\overline{\Omega}))\cap \mathcal C^{1+\alpha/2}([0,T_1(r+\epsilon)];\mathcal C(\overline{\Omega}))$ is continuous. Now let us consider
$$z=\frac{v_{1,\lambda+\delta}^{(1)}-v_{1,\lambda}^{(1)}}{\delta}-v_{1,\lambda}^{(2)}.$$
It is clear that $z$ solves the boundary value problem
\bel{pl2d}
\left\{
\begin{array}{ll}
\partial_tz-\Delta z+V_{1,\lambda}(t,x)z=L_{\delta,\lambda}  & \mbox{in}\ (0,T_1(r+\epsilon))\times\Omega ,
\\
z=0 &\mbox{on}\ (0,T_1(r+\epsilon))\times\partial\Omega,\\
z(0,x)=0  &x\in\Omega,
\end{array}
\right.
\ee
with
$$L_{\delta,\tau}=\partial_u^2F_1(\cdot,v_{1,\lambda})\left(v_{1,\lambda}^{(1)}\right)^2+\frac{V_{1,\lambda}-V_{1,\lambda+\delta}}{\delta}v_{1,\lambda+\delta}^{(1)}.$$
On the other hand, we have
$$ L_{\delta,\tau}=\partial_u^2F_1(\cdot,v_{1,\lambda})\left(v_{1,\lambda}^{(1)}\right)^2-\left(\int_0^1\partial_u^2F_1(\cdot,s v_{1,\lambda+\delta}+(1-s)v_{1,\lambda})ds\right)\left(\frac{v_{1,\lambda+\delta}-v_{1,\lambda}}{\delta}\right)v_{1,\lambda+\delta}^{(1)}.$$
In addition, the continuity of the map $(-r,r)\ni\lambda\longmapsto v_{1,\lambda}^{(1)}\in\mathcal C([0,T_1(r+\epsilon)];\mathcal C^{2+\alpha}(\overline{\Omega}))\cap \mathcal C^{1+\alpha/2}([0,T_1(r+\epsilon)];\mathcal C(\overline{\Omega}))$ and Lemma \ref{ppl1} imply that
\bel{pl2e}\lim_{\delta\to0}\norm{\frac{v_{1,\lambda+\delta}-v_{1,\lambda}}{\delta}-v_{1,\lambda}^{(1)}}_{W^{1,\infty}((0,T_1(r+\epsilon))\times\Omega)}=0,\ee
\bel{pl2f}\lim_{\delta\to0}\norm{v_{1,\lambda+\delta}^{(1)}-v_{1,\lambda}^{(1)}}_{W^{1,\infty}((0,T_1(r+\epsilon))\times\Omega)}=0.\ee
Applying the mean value theorem and Lemma \ref{ppl1}, we deduce that
$$\begin{aligned}&\norm{\int_0^1\partial_u^2F_1(\cdot,s v_{1,\lambda+\delta}+(1-s)v_{1,\lambda})ds-\partial_u^2F_1(\cdot,v_{1,\lambda})ds}_{W^{1,\infty}((0,T_1(r+\epsilon))\times\Omega)}\\
&\leq \int_0^1\norm{\partial_u^2F_1(\cdot,s v_{1,\lambda+\delta}+(1-s)v_{1,\lambda})-\partial_u^2F_1(\cdot,v_{1,\lambda})}_{W^{1,\infty}((0,T_1(r+\epsilon))\times\Omega)}ds\\
&\leq  \left(\max_{|\beta|+k\leq 1}\sup_{x\in\Omega}\sup_{\tau\in[-M_1(\lambda),M_1(\lambda)]}|\partial_u^3 \partial_x^\beta\partial_t^k F_1(t,x,\tau)|+|\partial_u^3  F_1(t,x,\tau)|\right)\\
&\ \ \ \times\norm{v_{1,\lambda+\delta}-v_{1,\lambda}}_{W^{1,\infty}((0,T_1(r+\epsilon))\times\Omega)}\\
&\ \ \ + M_1(\lambda)\left(\sup_{x\in\Omega}\sup_{\tau\in[-M_1(\lambda),M_1(\lambda)]}|\partial_u^4  F_1(t,x,\tau)|\right)\norm{v_{1,\lambda+\delta}-v_{1,\lambda}}_{L^\infty((0,T_1(r+\epsilon))\times\Omega)}.\end{aligned}$$
It follows that
$$\lim_{\delta\to0}\norm{\int_0^1\partial_u^2F_1(\cdot,s v_{1,\lambda+\delta}+(1-s)v_{1,\lambda})ds-\partial_u^2F_1(\cdot,v_{1,\lambda})}_{W^{1,\infty}((0,T_1(r+\epsilon))\times\Omega)}=0.$$
Combining this with \eqref{pl2e}-\eqref{pl2f}, we deduce that
$$\lim_{\delta\to0}\norm{L_{\delta,\lambda}}_{W^{1,\infty}((0,T_1(r+\epsilon))\times\Omega)}=0$$
and applying again \cite[Theorem 5.2, p.p. 320]{LSU}, we get
$$\norm{\frac{v_{1,\lambda+\delta}^{(1)}-v_{1,\lambda}^{(1)}}{\delta}-v_{1,\lambda}^{(2)}}_{2+\alpha,1+\alpha/2}=\norm{z}_{2+\alpha,1+\alpha/2}\leq C\norm{L_{\delta,\lambda}}_{W^{1,\infty}((0,T_1(r+\epsilon))\times\Omega)}.$$
This last inequality clearly implies \eqref{ppl2a}.
\end{proof}
\subsection{Completion of the proof of Theorem \ref{t1}}
Using the second linearization  of our problem given by Lemma \ref{ppl2}, we will complete the proof of Theorem \ref{t1} by showing that there  exists a constant $a_1>0$, depending only on $\Omega$, $T$, $q$, $\delta_1$ and $\chi$, such that $(\delta_1,T_1(r+\epsilon))\times\Omega\times(-a_1r,a_1r)\subset G$. Here $q\in L^\infty((0,T)\times\Omega)$ is given by \eqref{t1d}. For this purpose, we fix $x_0\in\Omega$, $t_0\in(\delta_1,T_1(r+\epsilon))$  and we will show that there  exists a constant $a_1>0$, depending only on $\Omega$, $T$, $q$, $\delta_1$ and $\chi$, such that the set
$G_{(t_0,x_0)}:=\{v_{1,\lambda}(t_0,x_0):\ \lambda\in(-r,r)\}$ contains the set $(-a_1r,a_1r)$. Since $\lambda\mapsto v_{1,\lambda}(t_0,x_0)\in\mathcal C([-r,r])$, the proof will be completed if we prove that
\bel{t5i}\pm v_{1,\pm r}(t_0,x_0)\geq\pm a_1r.\ee

We start by proving \eqref{t5i} for $\pm$ replaced by $+$. We fix $\lambda\in [-r,r]$ and we consider  $M_\lambda=\norm{V_{1,\lambda}}_{ L^\infty((0,T_1(r+\epsilon))\times\Omega)}$ and $\widetilde{v_{1,\lambda}}^{(1)}=e^{-M_\lambda t}v_{1,\lambda}^{(1)}$. One can check that $\widetilde{v_{1,\lambda}}^{(1)}$ satisfies
$$\left\{
\begin{array}{ll}
\partial_t\widetilde{v_{1,\lambda}}^{(1)}-\Delta \widetilde{v_{1,\lambda}}^{(1)}+(V_{1,\lambda}+M_\lambda)\widetilde{v_{1,\lambda}}^{(1)}=0  & \mbox{in}\ (0,T_1(r+\epsilon))\times\Omega,
\\
\widetilde{v_{1,\lambda}}^{(1)}(t,x)=e^{-M_\lambda t}\chi(t,x)  &(t,x)\in(0,T_1(r+\epsilon))\times\partial\Omega,\\
\widetilde{v_{1,\lambda}}^{(1)}(0,x)=0  &x\in\Omega.
\end{array}
\right.$$

Using the fact that $V_{1,\lambda}+M_\lambda\geq0$, $\chi\geq0$ and applying the weak maximum principle for parabolic equations (e.g. \cite[Theorem 9, pp. 369]{Ev}) to $\widetilde{v_{1,\lambda}}^{(1)}$ we deduce  that $$ v_{1,\lambda}^{(1)} (t,x)=e^{M_\lambda t}\widetilde{v_{1,\lambda}}^{(1)}(t,x)\geq0,\quad \lambda\in[-r,r],\  x\in\Omega,\ t\in(0,T_1(r+\epsilon)).$$ Combining this with Lemma \ref{ppl1} and the fact that, thanks to \eqref{t1a}, $v_{1,\lambda}|_{\lambda=0}=0$,  we find $$ v_{1,\lambda}(t,x)\geq 0,\quad \lambda\in[0,r],\ x\in\Omega,\ t\in(0,T_1(r+\epsilon)).$$
Thus, in light of \eqref{t1b}, we have
$$\partial_u^2F_1(t,x,v_{1,\lambda}(t,x))\leq 0,\quad \lambda\in[0,r],\ x\in\Omega,\ t\in(0,T_1(r+\epsilon)).$$
Then, in view of \eqref{ppeq6}, the weak maximum principle for parabolic equations (e.g. \cite[Theorem 9, pp. 369]{Ev}) combined with the above argumentation  imply that
$$v_{1,\lambda}^{(2)}(t,x)\geq 0,\quad \lambda\in[0,r],\ x\in\Omega,\ t\in(0,T_1(r+\epsilon)T).$$
Combining this with Lemma \ref{ppl2}, we deduce that
$$v_{1,\lambda}^{(1)} (t_0,x_0)\geq v_{1,0}^{(1)}(t_0,x_0),\quad \lambda\in[0,r].$$
and  we find
\bel{t5k}v_{1,\lambda} (t_0,x_0)=\int_0^\lambda v_{1,\tau}^{(1)} (t_0,x_0)d\tau \geq v_{1,0}^{(1)}(t_0,x_0)\lambda,\quad \lambda\in[0,r].\ee
In view of \eqref{t5k}, the proof will be completed if we show that there exists a constant $a_1>0$, depending only on $\Omega$, $T$, $q$, $\delta_1$ and $\chi$, such that
\bel{t1j} v_{1,0}^{(1)}(t_0,x_0)\geq a_1.\ee
For this purpose, let us fix $w\in H^1(0,T;H^{-1}(\Omega))\cap L^2(0,T;H^1(\Omega))$ solving
\bel{ww}\left\{
\begin{array}{ll}             
\partial_tw-\Delta w+qw=0  & \mbox{in}\ (0,T)\times\Omega,
\\
w(t,x)=\chi(t,x)  &(t,x)\in(0,T)\times\partial\Omega,\\
w(0,x)=0  &x\in\Omega.
\end{array}
\right.\ee
Since $q\in W^{1,\infty}((0,T)\times\Omega)$, \cite[Theorem 5.4, pp. 322]{LSU} implies that
$w\in \mathcal C([0,T];\mathcal C^{2+\alpha}(\overline{\Omega}))\cap \mathcal C^{1+\alpha/2}([0,T];\mathcal C(\overline{\Omega}))$.  We fix $w_1=v_{1,0}^{(1)}-w$ and we notice that $w_1$ solves 
$$\left\{
\begin{array}{ll}             
\partial_tw_1-\Delta w_1+V_{1,0}w_1=(q-V_{1,0})w  & \mbox{in}\ (0,T_1(r+\epsilon))\times\Omega,
\\
w_1(t,x)=0  &(t,x)\in(0,T_1(r+\epsilon))\times\partial\Omega,\\
w_1(0,x)=0  &x\in\Omega.
\end{array}
\right.$$
On the other hand, we have
$$q(t,x)-V_{1,0}(t,x)=q(t,x)-\partial_uF_1(t,x,v_{1,0}(t,x)),\quad (t,x)\in(0,T_1(r+\epsilon))\times\Omega$$
and from \eqref{t1a} we deduce that $v_{1,0}=0$. Thus, condition \eqref{t1d} implies that
$$q(t,x)-V_{1,0}(t,x)=q(t,x)-\partial_uF_1(t,x,0)\geq 0,\quad (t,x)\in(0,T_1(r+\epsilon))\times\Omega.$$
Moreover, since $\chi\geq0$, applying the weak maximum principle we deduce that $w\geq0$. It follows that $(q-V_{1,0})w\geq0$ and the weak maximum principle implies that $v_{1,0}^{(1)}-w=w_1\geq0$. Therefore, we have $v_{1,0}^{(1)}(t_0,x_0)\geq w(t_0,x_0)$ and the estimate \eqref{t1j} holds true with
\bel{a1}a_1:=\inf_{(t,x)\in (\delta_1,T)\times\Omega}w(t,x).\ee

Therefore, in order to complete the proof of the theorem we only need to show that $a_1>0$. For this purpose, let us assume the contrary. 
Since $w\in \mathcal C([\delta_1,T]\times\overline{\Omega})$, there exists $(t_1,x_1)\in [\delta_1,T]\times\overline{\Omega}$ such that $a_1=w(t_1,x_1)$. Since $w\geq0$, we have $a_1=w(t_1,x_1)=0$. Using the fact that 
$$w(t,x)=\chi(t,x)=\delta_2>0,\quad (t,x)\in [\delta_1,T]\times\partial\Omega,$$
we deduce that $x_1\in\Omega$. Thus, we can fix $\delta_3>0$ such that, for $B(x_1,\delta_3):=\{x\in\R^n:\ |x-x_1|<\delta_3\}$, we have $\overline{B(x_1,\delta_3)}\subset \Omega$ and we set $\delta_4\leq \min(\delta_1/2,(t_1-\delta_1)/2)$. Using the fact that $w\geq0$ and applying the parabolic Harnack inequality (e.g. \cite[Theorem 10 pp. 370]{Ev}), we deduce that, for all $t_2\in (\delta_1-\delta_4,\delta_1+\delta_4)$, we get
$$\sup_{x\in B(x_1,\delta_3)}w(t_2,x)\leq C\inf_{x\in B(x_1,\delta_3)}w(t_1,x)=0,$$
where $C$ depends on $t_1$, $t_2$, $x_1$, $\delta_3$, $\Omega$ and $q$. Therefore, we have
\bel{t5m}w(t,x)=0,\quad x\in B(x_1,\delta_3),\ t\in (\delta_1-\delta_4,\delta_1+\delta_4).\ee
Using the fact that $\Omega$ is connected and applying results of unique continuation for parabolic equations (e.g.  \cite[Theorem 1.1]{SS}), we deduce from \eqref{t5m}  that 
$$w(t,x)=0,\quad x\in \Omega,\ t\in (\delta_1-\delta_4,\delta_1+\delta_4).$$
Combining this with \eqref{ppeq5}, for any $x_2\in\partial\Omega$, we deduce that $\chi(\delta_1,x_2)=w(\delta_1,x_2)=0$ which contradicts the fact that $\chi(\delta_1,x_2)=\delta_2>0$. Thus, we have $a_1>0$ and the proof of \eqref{t5i} for $\pm$ replaced by $+$ is completed.

Now let us consider\eqref{t5i} for $\pm$ replaced by $-$. Repeating the above argumentation and applying \eqref{t1b}, we deduce that
$$\partial_u^2F_1(t,x,v_{1,\lambda}(t,x))\geq 0,\quad \lambda\in(-r,0),\ x\in\Omega,\ t\in(0,T_1(r+\epsilon)).$$
Applying the weak maximum principle for parabolic equations, we get
$$v_{1,\lambda}^{(2)}(t,x)\leq 0,\quad \lambda\in(-r,0),\ x\in\Omega,\ t\in(0,T_1(r+\epsilon))$$
and Lemma \ref{ppl2} implies that
$$v_{1,\lambda}^{(1)} (t_0,x_0)\geq v_{1,0}^{(1)}(t_0,x_0),\quad \lambda\in(-r,0).$$
Thus, we find
\bel{t1z}v_{1,\lambda} (t_0,x_0)=\int_0^\lambda v_{1,\tau}^{(1)} (t_0,x_0)d\tau\leq v_{1,0}^{(1)}(t_0,x_0)\lambda,\quad \lambda\in[-r,0],\ee
which, combined with  \eqref{t1j},  imply \eqref{t5i}. Therefore, for all $x_0\in\Omega$, $t_0\in(\delta_1,T_1(r+\epsilon))$ we have $[-a_1r_1,a_1r]\subset G_{(t_0,x_0)}$ and we deduce that $(\delta_1,T_1(r+\epsilon))\times\Omega\times[-a_1r_1,a_1r]\subset G$. This completes the proof of  Theorem \ref{t1}.

\section{Stable recovery of semilinear term}

This section is devoted to the proof of the stability result stated in Theorem \ref{t2} associated with the uniqueness result stated in Theorem \ref{t1}. For this purpose, we assume that the conditions of Theorem \ref{t1} are fulfilled and we would like to prove \eqref{t2b}. We start by considering the following intermediate result.

\begin{lem}\label{l4} Let the conditions of Theorem \ref{t1} be fulfilled, condition \eqref{t2a} be fulfilled, $r>0$ and consider $v_{1,\lambda}$, $\lambda\in[-r,r]$ the solution of \eqref{ppeq4}. Then, there exists $a_2>0$, depending only on $\Omega$, $T$, $\delta_1$, $\chi$ and $\kappa(0)$, such that, for any $(t,x)\in(\delta_1,T_1(r+\epsilon))\times\Omega$ and any $s\in[-a_2r,a_2r]$, there exists $\lambda_{t,x,s}\in[-r,r]$ with  $v_{1,\lambda_{t,x,s}}(t,x)=s$ \end{lem}
\begin{proof} Let us fix $(t,x)\in(\delta_1,T_1(r+\epsilon))\times\Omega$ and recall that the map $\lambda\mapsto v_{1,\lambda}(t,x)\in\mathcal C([-r,r])$. Thus, the proof of the lemma will be completed if we show that 
\bel{l4a}\pm v_{1,\pm r}(t,x)\geq a_2r.\ee
Let us first consider the above estimate with $\pm$ replaced by $+$. In view of \eqref{t5k}, the proof of \eqref{l4a}  will be completed if we show that there exists a constant $a_2>0$, depending only on $\Omega$, $T$, $\kappa(0)$, $\delta_1$ and $\chi$, such that
\bel{l4b} v_{1,0}^{(1)}(t_0,x_0)\geq a_2.\ee
For this purpose, let us fix $y \in \mathcal C([0,T];\mathcal C^{2+\alpha}(\overline{\Omega}))\cap \mathcal C^{1+\alpha/2}([0,T];\mathcal C(\overline{\Omega}))$ solving
\bel{eqy}\left\{
\begin{array}{ll}             
\partial_ty-\Delta y+\kappa(0)y=0  & \mbox{in}\ (0,T)\times\Omega,
\\
y(t,x)=\chi(t,x)  &(t,x)\in(0,T)\times\partial\Omega,\\
y(0,x)=0  &x\in\Omega.
\end{array}
\right.\ee
We set $z=v_{1,0}^{(1)}-y$ and we notice that $z$ solves 
$$\left\{
\begin{array}{ll}             
\partial_tz-\Delta z+\kappa(0)z=(\kappa(0)-V_{1,0})y  & \mbox{in}\ (0,T_1(r+\epsilon))\times\Omega,
\\
z(t,x)=0  &(t,x)\in(0,T_1(r+\epsilon))\times\partial\Omega,\\
z(0,x)=0  &x\in\Omega.
\end{array}
\right.$$
On the other hand, we have
$$\kappa(0)-V_{1,0}(t,x)=\kappa(0)-\partial_uF_1(t,x,v_{1,0}(t,x)),\quad (t,x)\in(0,T_1(r+\epsilon))\times\Omega$$
and from \eqref{t1a} we find $v_{1,0}=0$. Thus, condition \eqref{t2a} implies that
$$\kappa(0)-V_{1,0}(t,x)=\kappa(0)-\partial_uF_1(t,x,0)\geq 0,\quad (t,x)\in(0,T_1(r+\epsilon))\times\Omega.$$
Moreover, since $\chi\geq0$, applying the weak maximum principle we deduce that $y\geq0$. It follows that $(\kappa(0)-V_{1,0})y\geq0$ and the weak maximum principle implies that $v_{1,0}^{(1)}-y=z\geq0$. Therefore, we have $v_{1,0}^{(1)}(t,x)\geq y(t,x)$ and the estimate \eqref{l4b} holds true with
\bel{l4c}a_2:=\inf_{(t,x)\in (\delta_1,T)\times\Omega}y(t,x).\ee
Here it is clear that $a_2$ depends only on $\Omega$, $T$, $\delta_1$, $\chi$ and $\kappa(0)$ and the proof of the lemma will be completed if we show that $a_2>0$. This last property can be deduced by using arguments similar to the ones used in the proof of Theorem \ref{t1} for showing that $a_1>0$. This proves \eqref{l4a} with $\pm$ replaced by $+$. In the same way, we can complete the proof of  \eqref{l4a} with $\pm$ replaced by $-$.\end{proof}

Applying Lemma \ref{l4}, we are now in position to complete the proof of Theorem \ref{t2}.

\textbf{Proof of Thorem \ref{t2}.}  Let $r>0$, fix $s\in[-a_2r,a_2r]$ and consider $(t,x)\in(0,T_1(r+\epsilon))\times\Omega$. Consider also $\lambda_{t,x,s}$ given by Lemma \ref{l4}. Then, we have
\bel{t4d} \begin{aligned}&\abs{F_1(t,x,s)-F_2(t,x,s)}\\
&=\abs{F_1(x,v_{1,\lambda_{t,x,s}}(t,x))-F_2(t,x,v_{1,\lambda_{t,x,s}}(t,x))}\\
&\leq \abs{F_1(t,x,v_{1,\lambda_{t,x,s}}(t,x))-F_2(x,v_{2,\lambda_{t,x,s}}(t,x))}+\abs{F_2(t,x,v_{1,\lambda_{t,x,s}}(t,x))-F_2(t,x,v_{2,\lambda_{t,x,s}}(t,x))}.\end{aligned}\ee
We fix
$$I:=\abs{F_1(t,x,v_{1,\lambda_{t,x,s}}(t,x))-F_2(x,v_{2,\lambda_{t,x,s}}(t,x))},$$
$$II:=\abs{F_2(t,x,v_{1,\lambda_{t,x,s}}(t,x))-F_2(t,x,v_{2,\lambda_{t,x,s}}(t,x))}.$$
Using  estimate \eqref{t4d}, we will complete the proof of the theorem by proving that the expressions $I$ and $II$ can be estimated by the right hand of \eqref{t2b}. We start with $I$. Recall that
\bel{t4e} \begin{aligned}I&=\abs{\int_0^{\lambda_{t,x,s}}\partial_uF_1(t,x,v_{1,\tau}(t,x))v^{(1)}_{1,\tau}(t,x)-\partial_uF_2(t,x,v_{2,\tau}(t,x))v^{(1)}_{2,\tau}(t,x)d\tau}\\
&\leq \int_{-|\lambda_{t,x,s}|}^{|\lambda_{t,x,s}|}\abs{V_{1,\tau}(t,x)v^{(1)}_{1,\tau}(t,x)-V_{2,\tau}(t,x)v^{(1)}_{2,\tau}(t,x)}d\tau\\
&\leq\int_{-r}^{r}\abs{V_{1,\tau}(t,x)-V_{2,\tau}(t,x)}\abs{v^{(1)}_{1,\tau}(t,x)}d\tau+\int_{-r}^{r}\abs{V_{2,\tau}(t,x)}\abs{v^{(1)}_{1,\tau}(t,x)-v^{(1)}_{2,\tau}(t,x)}d\tau .\end{aligned}\ee
In view of condition (H), applying \eqref{t2a}, we obtain
\bel{t4g}\begin{aligned}&\norm{V_{j,\tau}}_{W^{1,\infty}((0,T_1(r+\epsilon))\times\Omega)}\\
&\leq \sup_{\lambda\in [-M(r+\epsilon),M(r+\epsilon)]}\norm{F_j(\cdot,\lambda)}_{W^{1,\infty}((0,T_1(r+\epsilon))\times\Omega)}+M(r+\epsilon)\sup_{\lambda\in [-M(r+\epsilon),M(r+\epsilon)]}\norm{\partial_uF_j(\cdot,\lambda)}_{W^{1,\infty}((0,T_1(r+\epsilon))\times\Omega)}\\
&\leq \kappa(M(r+\epsilon))(M(r+\epsilon)+1),\quad \tau\in[-r,r].\end{aligned}\ee
Combining \eqref{t4g} with \eqref{t4e}, we obtain
\bel{t4h}\begin{aligned}I\leq& 2r M(r+\epsilon) \sup_{\tau\in[-r,r]}\norm{V_{1,\tau}-V_{2,\tau}}_{L^\infty((0,T_1(r+\epsilon))\times\Omega)}\\
&+2r (M(r+\epsilon)+1)\kappa(M(r+\epsilon))\sup_{\tau\in[-r,r]}\norm{v_{1,\tau}^{(1)}-v_{2,\tau}^{(1)}}_{L^\infty((0,T_1(r+\epsilon))\times\Omega)}\end{aligned}.\ee
On the other hand, using \eqref{t4g} and considering the stability estimate of \cite[Theorem 1.1]{CK} with the geometric optics solutions of \cite[Proposition 4.3, 4.4]{CK1}, we get
$$\norm{V_{1,\lambda}-V_{2,\lambda}}_{L^2((0,T_1(r+\epsilon))\times\Omega)}\leq C_r\ln\left(3+\norm{\Lambda_{V_{1,\lambda}}-\Lambda_{V_{2,\lambda}}}_{\mathcal B\left(\mathcal K_0^{T_1(r+\epsilon)};\mathcal K_{T_1(r+\epsilon)}^*\right)}\right)^{-\theta_1},\quad \lambda\in[-r,r],$$
where $C_r>0$ depends on $r$, $\Omega$, $\kappa$, $M(r+\epsilon)$, $T_1(r+\epsilon)$, $\chi$, $n$ and $\theta_1>0$ depends on $n$. Recalling that $\mathcal N_{\lambda, r,F}'(0)=\Lambda_{V_{j,\lambda}}$ and applying a classical interpolation result (see e.g. \cite[Lemma AppendixB.1.]{CK}) as well as \eqref{t4g}, for all  $\lambda\in[-r,r]$, we obtain
$$\begin{aligned}\norm{V_{1,\lambda}-V_{2,\lambda}}_{L^\infty((0,T_1(r+\epsilon))\times\Omega)}&\leq C\norm{V_{1,\lambda}-V_{2,\lambda}}_{\mathcal C^{\frac{1}{2}}([0,T_1(r+\epsilon)]\times\overline{\Omega})}^{\frac{n+1}{n+2}}\norm{V_{1,\tau}-V_{2,\tau}}_{L^2((0,T_1(r+\epsilon))\times\Omega)}^{\frac{1}{n+2}}\\
\ &\leq C_r\ln\left(3+\norm{\mathcal N_{\lambda, r,F_1}'(0)-\mathcal N_{\lambda, r,F_2}'(0)}_{\mathcal B\left(\mathcal K_0^{T_1(r+\epsilon)};\mathcal K_{T_1(r+\epsilon)}^*\right)}\right)^{-\theta},\end{aligned}$$
with  $C_r>0$ depending on $r$, $\Omega$, $T$, $\kappa$, $M(r+\epsilon)$, $T_1(r+\epsilon)$, $\chi$, $n$ and with with $\theta>0$ depending on $n$. It follows that
\bel{t4i}\begin{aligned}&\sup_{\lambda\in(-r,r)}\norm{V_{1,\lambda}-V_{2,\lambda}}_{L^\infty((0,T_1(r+\epsilon))\times\Omega)}\\
&\leq C_r\sup_{\lambda\in(-r,r)}\ln\left(3+\norm{\mathcal N_{\lambda, r,F_1}'(0)-\mathcal N_{\lambda, r,F_2}'(0)}_{\mathcal B\left(\mathcal K_0^{T_1(r+\epsilon)};\mathcal K_{T_1(r+\epsilon)}^*\right)}\right)^{-\theta}.\end{aligned}\ee
According to the weak maximum principle, we find 
\bel{t4j}\norm{v_{2,\lambda}^{(1)}}_{L^\infty((0,T_1(r+\epsilon))\times\Omega)}\leq 1,\quad \lambda\in[-r,r].\ee
Fixing $w=v_{1,\lambda}^{(1)}-v_{2,\lambda}^{(1)}$, we deduce that $w$ solves
$$\left\{
\begin{array}{ll}
\partial_tw-\Delta w+V_{1,\lambda}w=(V_{2,\lambda}-V_{1,\lambda})v_{2,\lambda}^{(1)}  & \mbox{in}\ (0,T_1(r+\epsilon))\times\Omega ,
\\
w=0 &\mbox{on}\ \partial(0,T_1(r+\epsilon))\times\Omega,\\
w(x,0)=0  &x\in\Omega.
\end{array}
\right.$$
Therefore, in view of \cite[Theorem 9.1, pp. 341]{LSU}, fixing $p>n+1$ and applying \eqref{t4g}, \eqref{t4j} as well as the Sobolev embedding theorem, we deduce that
$$\begin{aligned}\norm{v_{1,\lambda}^{(1)}-v_{2,\lambda}^{(1)}}_{L^\infty((0,T_1(r+\epsilon))\times\Omega)}&\leq C\norm{w}_{W^{1,p}((0,T_1(r+\epsilon))\times\Omega)}\\
\ &\leq C_r\norm{V_{1,\lambda}^{(1)}-V_{2,\lambda}^{(1)}}_{L^\infty((0,T_1(r+\epsilon))\times\Omega)},\quad \lambda\in[-r,r].\end{aligned}$$
Then, \eqref{t4i} implies
\bel{t4k}\begin{aligned}&\sup_{\lambda\in(-r,r)}\norm{v_{1,\lambda}^{(1)}-v_{2,\lambda}^{(1)}}_{L^\infty((0,T_1(r+\epsilon))\times\Omega)}\\
&\leq C_r\sup_{\lambda\in(-r,r)}\ln\left(3+\norm{\mathcal N_{\lambda, r,F_1}'(0)-\mathcal N_{\lambda, r,F_2}'(0)}_{\mathcal B\left(\mathcal K_0^{T_1(r+\epsilon)};\mathcal K_{T_1(r+\epsilon)}^*\right)}\right)^{-\theta}.\end{aligned}\ee
Combining this with \eqref{t4h} and \eqref{t4i}, we get
\bel{t4l}I\leq C_r\sup_{\lambda\in(-r,r)}\ln\left(3+\norm{\mathcal N_{\lambda, r,F_1}'(0)-\mathcal N_{\lambda, r,F_2}'(0)}_{\mathcal B\left(\mathcal K_0^{T_1(r+\epsilon)};\mathcal K_{T_1(r+\epsilon)}^*\right)}\right)^{-\theta}.\ee
Now let us consider $II$. Applying  the mean value theorem and \eqref{t2a}, we get
$$\begin{aligned}II&\leq \sup_{\lambda\in[-M(r+\epsilon),M(r+\epsilon)]}\norm{\partial_u F_2(\cdot,\lambda)}_{L^\infty((0,T_1(r+\epsilon))\times\Omega)}\abs{v_{1,\lambda_{t,x,s}}(t,x)-v_{2,\lambda_{t,x,s}}(t,x)}\\
&\leq \kappa(M(r+\epsilon))2r\sup_{\lambda\in(-r,r)}\norm{v_{1,\lambda}^{(1)}-v_{2,\lambda}^{(1)}}_{L^\infty((0,T_1(r+\epsilon))\times\Omega)}.\end{aligned}$$
Then, applying \eqref{t4k}, we deduce that
$$II\leq C_r\sup_{\lambda\in(-r,r)}\ln\left(3+\norm{\mathcal N_{\lambda, r,F_1}'(0)-\mathcal N_{\lambda, r,F_2}'(0)}_{\mathcal B\left(\mathcal K_0^{T_1(r+\epsilon)};\mathcal K_{T_1(r+\epsilon)}^*\right)}\right)^{-\theta}.$$
Combining this estimate with \eqref{t4l} and \eqref{t4d} we deduce \eqref{t2b}. This completes the proof of Theorem \ref{t2}.\qed

\section*{Acknowledgements}

The work of Y.K is partially supported by  the French National Research Agency ANR (project MultiOnde) grant ANR-17-CE40-0029.
 The research of G.U. is partially supported by NSF, a Walker Professorship at UW and a Si-Yuan Professorship at IAS, HKUST. Part of the work was supported by the NSF grant DMS-1440140 while  G.U. were in residence at MSRI in Berkeley, California, during Fall 2019 semester.

\bigskip
\vskip 1cm

\end{document}